\documentclass{amsart} 
\usepackage{amsfonts,amssymb,amsmath,amsthm} 
\usepackage{url} 
\usepackage{enumerate} 
\usepackage{stmaryrd}

\newtheorem{thrm}{Theorem}[section]

\newtheorem{cor}[thrm]{Corollary}
\theoremstyle{definition}
\newtheorem{definition}[thrm]{Definition}
\newtheorem{remark}[thrm]{Remark}

\numberwithin{equation}{section}

\newcommand{\M}{{M}}

\newcommand{\U}{{U}}

\newcommand{\csm}{\mathrm{csm}}

\makeatletter
\let\@@citation@@=\citation

\renewcommand{\citation}[1]{\@@citation@@{#1}%
\@for\@tempa:=#1\do{\@ifundefined{cit@\@tempa}%
  {\global\@namedef{cit@\@tempa}{}}{}}%
}

\def\@lbibitem[#1]#2#3\par{%
  \@ifundefined{cit@#2}{}{\item[\@biblabel{#1}\hfill]}%
  \if@filesw
      {\let\protect\noexpand
       \immediate
       \write\@auxout{\string\bibcite{#2}{#1}}}\fi\ignorespaces
  \@ifundefined{cit@#2}{}{#3}}
\def\@bibitem#1#2\par{%
  \@ifundefined{cit@#1}{}{\item}%
  \if@filesw \immediate\write\@auxout
    {\string\bibcite{#1}{\the\value{\@listctr}}}\fi\ignorespaces
  \@ifundefined{cit@#1}{}{#2}}
\makeatother

\author[Ashraf]{Ahmed Umer Ashraf}
\address{Mathematics Department\\
 University of Western Ontario\\
 London, Ontario, Canada.}
\email{aashra9@uwo.ca}
\author[Backman]{Spencer Backman}
\address{Department of Mathematics and Statistics,
University of Vermont, Burlington,
Vermont, USA.}
\email{Spencer.Backman@uvm.edu}
\thanks{}

\keywords{Matroids, tropical intersection theory, Tutte polynomial, Tutte activities, matroid Chern-Schwartz-MacPherson cycles}
\subjclass{Primary 05B35 , Secondary 52B40}
\begin{document}

\title{Matroid Chern-Schwartz-MacPherson cycles and Tutte activities}

\begin{abstract}
 \hspace{-0.7em} L\'opez de Medrano-Rinc\'on-Shaw defined Chern-Schwartz-MacPher-son  cycles for an arbitrary matroid $\M$ and proved by an inductive geometric argument that the unsigned degrees of these cycles agree with the coefficients of $T(\M;x,0)$, where $T(\M;x,y)$ is the Tutte polynomial associated to $\M$.  Ardila-Denham-Huh recently utilized this interpretation of these coefficients in order to demonstrate their log-concavity.  In this note we provide a direct calculation of the degree of a matroid Chern-Schwartz-MacPherson  cycle by taking its stable intersection with a generic tropical linear space of the appropriate codimension and showing that the weighted point count agrees with the Gioan-Las Vergnas refined activities expansion of the Tutte polynomial.
\end{abstract}
\maketitle
\section{The $\beta$-expansion of the Tutte polynomial}

Let $E$ be a finite set equipped with a linear ordering $<$, and let $\M$ be a matroid on $E$. Throughout this note, we fix the underlying set to be $E = \llbracket n \rrbracket := \{ 0,1, \ldots, n\}$ equipped with the natural order. We also fix the rank of the matroid $M$ to be $d+1 > 0$. 
We denote the collection of bases of $\M$ by $\mathcal{B}(\M)$. Consider a basis $B$ of $\M$. For an element $e \in  \llbracket n \rrbracket \backslash B$, the \emph{fundamental circuit} $\gamma(e;B)$ of $B$ with respect to $e$ is given by
\begin{align}
    \gamma(e; B) &:= \{e' \in E: B \cup e \backslash e' \in \mathcal{B}(\M)\}
\end{align}
and for an element $e \in B$, the \emph{fundamental cocircuit} $\gamma^*(e;B)$ of $B$ with respect to $e$ is given by
\begin{align}
    \gamma^*(e; B) &:= \{e' \in E: B \cup e' \backslash e \in \mathcal{B}(\M)\}.
\end{align}
The \emph{external activity} $\mathrm{ex}(B)$ and \emph{internal activity} $\mathrm{in}(B)$ of a basis $B$ are defined as
\begin{align}
    \mathrm{ex}(B) &:= \big| \{ e \in \llbracket n \rrbracket \backslash B: \min_< \gamma(e; B) = e  \} \big|, \\
    \mathrm{in}(B) &:= \big| \{ e \in  B: \min_< \gamma^*(e; B) = e  \} \big|. 
\end{align}
The Tutte polynomial of a matroid $\M$ is the generating function for these statistics over the collection of bases of $\M$:
\begin{align}
    T(\M; x, y) = \sum_{B \in \mathcal{B}(\M)} x^{\mathrm{in}(B)}y^{\mathrm{ex}(B)}=\sum_{i,j} t_{i,j}x^iy^j.
\end{align}
We note that this polynomial is well-defined, i.e. it is independent of the total order on the ground set. We denote by $\mathcal{B}_{i,j}(M)$ the set of bases $B$ of the matroid $M$ with $\mathrm{ex}(B)=i$ and $\mathrm{in}(B)=j$  so that $|\mathcal{B}_{i,j}(M)|=t_{i,j}$.  We refer the reader to \cite{backman-activities} for an overview of the theory of Tutte activities.

Gioan and Las Vergnas gave a refined activities expansion of the Tutte polynomial. We recall Crapo's beta invariant $\beta(\M):=t_{1,0}(\M)$.  Note that for a nonempty matroid $\M$ which is not a loop,  $\beta(\M) = 0$ if and only if $\M$ is not connected.  If $\mathcal{F}$ is a flag of flats of a loopless matroid $\M$ of the form
\begin{align*}
    \mathcal{F}= \{ \varnothing = F_0 \subsetneq F_1 \subsetneq \cdots \subsetneq  F_k  = \llbracket n \rrbracket \}
\end{align*}
where $F_i$ is a flat of $M$ for each $i=0, \ldots, k$, then we call $\mathcal{F}$ \emph{proper}. The integer $k$ is the \emph{length} of the proper flag $\mathcal{F}$, and we denote it by $\ell(\mathcal{F})$. The difference $F_i\backslash F_{i-1}$ will be denoted by $\Delta_i$ whenever the flag $\mathcal{F}$ is clear from the context. We call a proper flag $\mathcal{F}$ \emph{increasing} if $  \min(\Delta_1)< \min(\Delta_2) < \cdots <\min(\Delta_{k})$. The one-variable specialization of the Tutte polynomial of a loopless matroid $M$ has the following expansion:
\begin{align} \label{eq:tutte-beta}
    T(\M;x, 0) &= \sum_\mathcal{F} \bigg[\prod_{i=1}^{\ell(\mathcal{F})} \beta(\M|F_{i}/F_{i-1}) \bigg] x^{\ell(\mathcal{F})}
\end{align}
where the sum is taken over all increasing proper flags of flats.  We may assume that each minor $M|F_i/F_{i-1}$ is connected,  otherwise $\beta(M|F_i/F_{i-1}) =0$. This formula is a direct consequence of the following structural result of Gioan and Las Vergnas in \cite{gioan-lasvergnas, gioan-thesis}.   

Let $\mathcal{A}_{k}(\M)$ be the set of all pairs $(\mathcal{F},\mathcal{B_F})$ where $\mathcal{F} = \{ \varnothing = F_0 \subsetneq F_1 \subsetneq \cdots \subsetneq F_{k} = \llbracket n \rrbracket\}$ is an increasing proper flag of flats of length $k$, and $\mathcal{B_F} = (B_1,\dots, B_{k})$ where $B_i \in \mathcal{B}_{1,0}(\M|F_i/F_{i-1})$ for $1 \leq i \leq k$.

\begin{thrm}[{\cite[Theorem 4.2]{gioan-lasvergnas} \cite[Theorem 2.3.1]{gioan-thesis}}]
  For $1 \leq k \leq d+1$, the sets $\mathcal{A}_{k}(\M)$ and $\mathcal{B}_{k,0}(\M)$  are canonically in bijection.
\end{thrm}

 Although not necessary for  this note, we remark that Gioan-Las Vergnas describe a more general activities decomposition for arbitrary bases, which refines an earlier decomposition of Etienne-Las Vergnas \cite[Theorem 5.1]{etienne-lasvergnas}, and implies a formula for the full Tutte polynomial.

\section{Elementary tropical intersection theory}

In this section we provide background on tropical intersection theory necessary for our main calculation in Section \ref{csmsection}.  We recommend \cite{fulton-sturmfels, maclagan-sturmfels, jensen-yu} for further details.

Let $\Sigma$ be a rational polyhedral fan of pure dimension $m$ in $\mathbb{R}^{n}$ and let $N \cong \mathbb{Z}^n$ be the lattice generated by a choice of basis vectors. We denote by $\Sigma_k$ its $k$-skeleton for $k \leq m$. For a cone $\tau$ of $\Sigma$, let $N_\tau$ denote the maximal sublattice of $N$ parallel to the linear span of $\tau$. Given such a fan $\Sigma$ and a weight function $c: \Sigma_m \rightarrow \mathbb{Z}$, we say that $\Sigma$ satisfies the \emph{balancing condition} at the ridge $\tau \in \Sigma_{m-1}$ with respect to weight $c$ if
\begin{align}
    \sum_{  \sigma \supsetneq \tau} c (\sigma) \mathbf{n}_{\sigma} &\in \text{span}_{\mathbb{R}}N_\tau ,
\end{align}
where $\mathbf{n}_{\sigma}$ is a primitive generator of $N_\sigma$ outside $N_\tau$, i.e. $\mathrm{span}_\mathbb{Z}(\mathbf{n}_{\sigma}, N_\tau) = N_\sigma$.
We say a pure rational fan $\Sigma$ of dimension $m$ is \emph{balanced} with respect to weight $c$ if it satisfies the balancing condition at each ridge $\tau \in \Sigma_{m-1}$. We say two balanced fans are equivalent if they induce the same weight function on their common refinement. An equivalence class of balanced fans is called a \emph{tropical cycle}. 

Let $\M$ be a loopless matroid of rank $d+1$ on the set $\llbracket n \rrbracket$. Let $\mathbf{e}_i$ denote the $i$th standard basis vector in $\mathbb{R}^{n+1}$, and for $F \subseteq \llbracket n \rrbracket$, let $\mathbf{e}_F = \sum_{i \in F} \mathbf{e}_i$. To each proper flag of flats
\begin{align}
    \mathcal{F} = \{ \varnothing = F_0 \subsetneq F_1 \subsetneq \cdots \subsetneq F_k  = \llbracket n \rrbracket\}
\end{align} 
we can associate a rational polyhedral cone $\sigma_\mathcal{F}$ in $\mathbb{R}^{n+1}$ given by
\begin{align}
    \sigma_\mathcal{F} = \mathbb{R}_{\geq 0} \{\mathbf{e}_{F_1}, \ldots, \mathbf{e}_{F_{k-1}}\} + \mathbb{R}\mathbf{e}_{\llbracket n \rrbracket}.
\end{align}
The \emph{non-reduced Bergman fan} $\widetilde{\Sigma}(\M)$ of $\M$ is the rational polyhedral fan consisting of the cones $\sigma_\mathcal{F}$ for each proper flag $\mathcal{F}$ of flats of $\M$. This is a pure $(d+1)$-dimensional simplicial fan in $\mathbb{R}^{n+1}$, which has $\mathbb{R} \mathbf{e}_{\llbracket n \rrbracket }$ as its lineality space. 
We define the \emph{Bergman fan} of $\M$, written $\Sigma(\M)$, to be the image of $\widetilde{\Sigma}(\M)$ in the quotient space $\mathbb{R}^{n+1}/\mathbb{R}\mathbf{e}_{\llbracket n \rrbracket}$ with the lattice $N = \mathbb{Z}^{n+1}/\mathbb{Z}\mathbf{e}_{\llbracket n \rrbracket}$. The Bergman fan of a matroid is a tropical cycle of dimension $d$: it is balanced with respect to the weight function $\mathbf{1}$ which assigns the weight 1 to each maximal cone. When no confusion may arise, we also denote this tropical cycle by $\Sigma(\M)$.


Given two tropical cycles $\mathcal{T} = (\Sigma, c)$ and $\mathcal{T}' = (\Sigma', c')$, their \emph{stable intersection} $\mathcal{T} \cdot \mathcal{T}'$ is defined as the tropical cycle with the support $\lim_{\mathbf{v} \rightarrow \mathbf{0}} |\Sigma| \cap |\Sigma' +  \mathbf{v}|$, and the weights given by 
\begin{align}
    m(\gamma) &= \sum_{\substack{\sigma \in \Sigma\\ \sigma' \in \Sigma'}} c(\sigma)c'(\sigma')[N: N_\sigma + N_{\sigma'}],
\end{align}
where the sum is taken over all $ \sigma, \sigma' \supseteq \gamma $ with $\sigma \cap (\sigma' +  \mathbf{v}) \neq \varnothing$ for any generic perturbation $\mathbf{v}$. The fan structure on $\mathcal{T} \cdot \mathcal{T}'$ is given by the common refinement of $\Sigma$ and $\Sigma'$.

Let $\U_{n+1-k, \llbracket n \rrbracket}$ be the uniform matroid on $\llbracket n \rrbracket$ of rank $n+1-k$, and note that $F \subsetneq \llbracket n \rrbracket$ is a flat of $\U_{n+1-k, \llbracket n \rrbracket}$ if and only if $|F| \leq n-k$.  Let $\mathcal{T} = (\Delta, c)$ be a $k$-dimensional tropical cycle in $\mathbb{R}^n$.  The  \emph{degree of $\mathcal{T}$}, written $\deg(\mathcal{T})$, is defined as $\sum_{\mathbf{p}} m(\mathbf{p})$, where the sum ranges over the finite set of points in the support of $\mathcal{T} \cdot \Sigma(\U_{n+1-k,\llbracket n \rrbracket})$. We note that the fan $\Sigma(\U_{n+1-k,\llbracket n \rrbracket})$ is considered a generic tropical linear space, thus this definition is a tropical analogue of the classical notion of the degree of a variety.

\section{Chern-Schwartz-MacPherson cycles and the main calculation}\label{csmsection}

Chern-Schwartz-MacPherson cycles for complements of complex hyperplane arrangements are certain matroid invariants which have been studied in \cite{cohen-denham-falk-varchenko,denham-garrousian-schulze, huh-chromatic, aluffi, huh-likelihood, huh-logconcavity}. L\'opez de Medrano-Rinc\'on-Shaw  provided a generalization of Chern-Schwartz-MacPherson cycles for arbitrary matroids.

 \begin{definition}{\cite[Definition 5]{ medrano-rincon-shaw}}\label{csmdef}
 Let $\M$ be a matroid of rank $d+1$.  The \emph{$k$th Chern-Schwartz-MacPherson cycle} $\csm_k(\M)$ for each $k=0, 1, \ldots, d$ is the weighted fan $(\Sigma_{k}, \omega)$, where $\Sigma_{k}$ is the $k$th skeleton of Bergman fan of $\M$ and $\omega: \Sigma_{k} \rightarrow \mathbb{Z}$ is given by
\begin{align}\label{csmform}
    \omega(\sigma_{\mathcal{F}}) &= (-1)^{d-k}\prod_{i=0}^k \beta(\M|F_{i+1}/F_i).
\end{align}

\end{definition}
 In \cite[Theorem 2.3]{medrano-rincon-shaw} it  was shown that the $\csm_k(M)$ cycles of a matroid of a hyperplane arrangement defined over $\mathbb{C}$ are given by the Chern-Schwartz-MacPherson cycles of the complement of the arrangement via the theory of Minkowski weights. Furthermore, it is demonstrated that these are tropical cycles. We remark that the presence of the $(-1)^{d-k}$ in Equation \eqref{csmform} is necessary for Definition \ref{csmdef} to agree with the definition for complements of complex hyperplane arrangements.  
 
 In the case of matroids representable over the complex numbers, additivity for Chern-Schwartz-MacPherson cycles has been applied to show that their degrees satisfy a specialization of the deletion-contraction formula for the Tutte polynomial \cite{cohen-denham-falk-varchenko,denham-garrousian-schulze, huh-chromatic, aluffi, huh-likelihood, huh-logconcavity}.   L\'opez de Medrano-Rinc\'on-Shaw combinatorially extended this inductive degree calculation to arbitrary matroids.
 


\begin{thrm}[{\cite[Theorem 1.4]{medrano-rincon-shaw}}] \label{thm:main}
Let $\M$ be a loopless matroid on $\llbracket n \rrbracket$ of rank $d+1$. Let $\csm_{k}(\M)$ denote its $k$th Chern-Schwartz-MacPherson cycle, then for $k= 0, \ldots, d$
\begin{align} 
    \deg(\csm_{k}(\M) ) = (-1)^{d-k}t_{k+1,0}.
\end{align}
\end{thrm}

This interpretation of the coefficients of $T(M;x,0)$ has recently been utilized by Ardila-Denham-Huh \cite{ardila-denham-huh} for proving their log-concavity generalizing an earlier result of Huh \cite{huh-logconcavity}. 

The coefficients of $T(M;x,0)$ have alternate interpretations which have been invoked in the literature related to Chern-Schwartz-MacPherson cycles for matroids.  The coefficients of $T(M;x,0)$ agree with the $h$-vector of the broken circuit complex of $\M$.  See  \cite{bjorner-shellability} for a shelling proof of this fact, which can be alternately established by application of the generalized activities expansion of the Tutte polynomial due to Gordon-Traldi \cite[Theorem 3]{gordon-traldi}. 

When $M$ is a loopless non-empty matroid, its reduced characteristic polynomial is given by 
\begin{align*}
    {\bar \chi}_{\M}(q) = \frac{\chi_{\M}(q)}{(q-1)} =  {(-1)^{d+1}\frac{T(\M;1-q,0)}{(q-1)}},
\end{align*}
hence the shifted reduced characteristic polynomial is 
\begin{align}
   {\bar \chi}_{\M}(q+1) = {(-1)^{d+1}\frac{T(\M;-q,0)}{ q}}. 
\end{align}
Theorem \ref{thm:main} is presented in \cite{medrano-rincon-shaw} in terms of the coefficients of  $\overline{\chi}_M(q+1)$.

The purpose of this note is to now provide a direct non-recursive proof of Theorem \ref{thm:main} which highlights a connection to the theory of Tutte activities via the work of Gioan and Las Vergnas \cite{gioan-thesis,gioan-lasvergnas}.

\begin{proof}[Proof of Theorem \ref{thm:main}]

We compute the stable intersection $\csm_{k}(M)\cdot \Sigma(\U_{n+1-k, \llbracket n \rrbracket})$.  All vectors considered in the proof lie in $\mathbb{R}^{n+1}/\mathbb{R}\mathbf{e}_{\llbracket n \rrbracket}$. Given some vector $\mathbf{v} \in \mathbb{R}^{n+1}$, let $[\mathbf{v}]$ be its image in $\mathbb{R}^{n+1}/\mathbb{R}\mathbf{e}_{\llbracket n \rrbracket}$.  Now fix some generic $\mathbf{v}\in \mathbb{R}^{n+1}$.  Let $\mathbf{p}, \mathbf{q} \in \mathbb{R}^{n+1}$ be such that $[\mathbf{p}] \in \csm_{k}(\M)$, and $\mathbf{v}+\mathbf{q}=\mathbf{p}$.  We will completely determine the conditions on $[\mathbf{p}]$ which ensure that $[\mathbf{q}] \in \Sigma(\U_{n+1-k, \llbracket n \rrbracket})$.  

Suppose that $\mathbf{p}$ lies in the maximal cone $\sigma_{\mathcal{F}}$ of $\csm_{k}(\M)$ associated to the flag of flats $\mathcal{F} = \{\varnothing = F_0 \subsetneq F_1 \subsetneq \cdots \subsetneq F_{k} \subsetneq F_{k+1} = \llbracket n \rrbracket\}$ of $\M$.  By the genericity of $\mathbf{v}$, we may assume after relabeling the indices that for every $0\leq i\leq n$ we have $\mathbf{v}_i > \mathbf{v}_{i+1}$.  Next, by adding the appropriate multiple of $\mathbf{e}_{\llbracket n \rrbracket}$ to $\mathbf{p}$ and $\mathbf{q}$, we may assume $\mathbf{p}\geq \mathbf{v}$ and there exists some $i$ such that $\mathbf{p}_i = \mathbf{v}_i$.  This gives that $\mathbf{q} \geq 0$ and there exists some $i$ with $\mathbf{q}_i = 0$.  Under these assumptions, $[\mathbf{q}] \in \Sigma(\U_{n+1-k, \llbracket n \rrbracket})$ if and only if $\mathbf{q}$ has support of size at most $n-k$.  

Let $1\leq m \leq k+1$, and recall that for any $i \in \Delta_m$, the entry $\mathbf{p}_i$ has the same value, which we denote $\mathbf{p}^m$.   We claim that $\mathbf{p}^m = \max_{a \in \Delta_m} \mathbf{v}_a $ for each $1\leq m \leq k+1$ if and only if $[\mathbf{q}] \in \Sigma(\U_{n+1-k, \llbracket n \rrbracket})$.  Because $\mathbf{v}$ is strictly decreasing, for each $1 \leq m \leq k$ there can be at most one index $i \in \Delta_m$ such that $\mathbf{p}^m=\mathbf{v}_i$.  Therefore if there exists an $m$ such that $\mathbf{p}^m > \max_{i \in \Delta_m} \mathbf{v}_i$, the support of $\mathbf{q}$ is strictly greater than $n-k$, thus $[\mathbf{q}] \notin \Sigma(\U_{n+1-k, \llbracket n \rrbracket})$.  Conversely, if $\mathbf{p}^m = \max_{i \in \Delta_m} \mathbf{v}_i $, because $\mathbf{v}$ is strictly decreasing we see that for $i \in \Delta_m$ we have $\mathbf{q}_i =0$ if and only if $i$ is the minimum index in $\Delta_m$, thus the support of $\mathbf{q}$ has size $n-k$ and $[\mathbf{q}] \in \Sigma(\U_{n+1-k, \llbracket n \rrbracket})$.  This calculation demonstrates that there is a single intersection point for each increasing proper flag of flats associated to a maximal cone of $\csm_{k}(M)$.

Let $\mathcal{I}$ be the maximal flag of flats of $U_{n+1-k,\llbracket n \rrbracket}$ such that $\sigma_\mathcal{F}$ intersects $\sigma_{\mathcal{I}}$, a cone in $\Sigma(U_{n+1-k, \llbracket n \rrbracket})$. Then the cone $\sigma_\mathcal{I}$ consists of all points $[\mathbf{q}] \in \mathbb{R}^{n+1}/\mathbb{R}\mathbf{e}_{\llbracket n \rrbracket}$ corresponding to $\mathbf{q} \in \mathbb{R}^n$ such that $\mathbf{q}_i = 0$ if and only if $i$ is the minimum index of $\Delta_m$ for $1 \leq m \leq k+1$. In other words $N_{\sigma_\mathcal{I}}$ is generated by $\mathbf{e}_i$'s for $i \in \llbracket n \rrbracket \backslash \cup_i \min (\Delta_i)$.
This implies that points in $N_{\sigma_\mathcal{I}}$ and $N_{\sigma_{\mathcal{F}}}$ together generates $N$. Hence the index $[N:N_{\sigma_\mathcal{I}} + N_{\sigma_{\mathcal{F}}}]=1$. Thus, the weight of a point $[\mathbf{p}]$ in the intersection is
\begin{align}
    m([\mathbf{p}]) &= (-1)^{d-k}\prod_{i=0}^k \beta(\M|F_{i+1}/F_i).
\end{align}
Therefore
\begin{align}
   \deg(\csm_{k}(\M) ) = \sum m([\mathbf{p}]) &= (-1)^{d-k}\sum \prod_{i=1}^k \beta(\M|F_{i}/F_{i-1}),
\end{align}
where the sum on the right is over all proper increasing flags of flats $\mathcal{F}$ of $\M$ of length $k$. By Equation \eqref{eq:tutte-beta}, this equals the coefficient $(-1)^{d-k}t_{k+1, 0}$. \end{proof}

\begin{cor}\label{welldef}
 Let $M$ be a loopless non empty matroid, then $T(M;x,0)$ is well-defined, i.e. it is independent of the total order on the ground set.
\end{cor}

\begin{proof}
In the proof of Theorem \ref{thm:main}, we always obtain the same value for the degree of $\csm_{k}(\M)$ regardless of which chamber of the braid arrangement into which we perturb our fan, and each such chamber corresponds to a different total order on the ground set.  Thus the well-definedness of $T(M;x,0)$ is a consequence of the well-definedness of the degree map, which is itself a consequence of the tropical balancing condition. 
\end{proof}

\begin{remark}
 The perspective offered in the proof of Corollary \ref{welldef} is implicit in the earlier work of Huh-Katz\cite{huh-katz} who use tropical intersection theory to show that the degrees of reciprocal linear spaces agree with the unsigned coefficients of the reduced characteristic polynomial.  For performing this calculation, those authors first demonstrate that these unsigned coefficients count \emph{initial $k$-step flags of flats}.  It is not hard to show that these flags are canonically in bijection with the  faces of the reduced broken circuit complex, i.e. the independent sets contained in bases in ${\mathcal B}_{k,0}(M)$ which avoid the element $0$.  We note that a more standard method for proving the well-definedness of the coefficients of the Tutte polynomial is to demonstrate Crapo's interval decomposition of the Boolean lattice on the ground set of $M$.  This decomposition implies that $T(M;x+1,y+1)$ is the corank-nullity generating function whose definition does not involve a total order on the ground set.  See \cite{backman-activities} for an overview of this decomposition.
\end{remark}

\section{Acknowledgements}

The authors thank Graham Denham, Alex Fink, Emeric Gioan, Felipe Rincon, and Kristin Shaw for inspiring discussions.


\begin{thebibliography}{}
\bibitem[ADH20]{ardila-denham-huh}
{Federico Ardila, Graham Denham, and June Huh},
\emph{Lagrangian geometry of matroids},
\url{https://arxiv.org/pdf/2004.13116.pdf}.


\bibitem[AHK18]{adiprasito-huh-katz}
{Adiprasito, Karim and Huh, June and Katz, Eric},
\emph{Hodge theory for combinatorial geometries},
{Ann. of Math. (2)},
{\bf 188},
{(2018)},
{no.~2},
\url{https://doi.org/10.4007/annals.2018.188.2.1}.

\bibitem[Alu13]{aluffi}
{Paolo Aluffi},
\emph{Grothendieck classes and {C}hern classes of hyperplane arrangements},
{Int. Math. Res. Not. IMRN},
{(2013)},
{no.~8},
{1873--1900},
\url{https://doi.org/10.1093/imrn/rns100}.

\bibitem[Ash19]{ashraf-thesis}
{Ahmed Umer Ashraf},
\emph{Of matroid polytopes, chow rings and character polynomials}
{PhD thesis},
{University of Western Ontario},
{(2019)}.

\bibitem[Bac20+]{backman-activities}
{Spencer Backman},
\emph{Tutte Activities}
In: J. Ellis-Monaghan and I. Moffatt, eds., Handbook on the Tutte Polynomial and Related Topics, CRC Press, (2020+).

\bibitem[Bj\"o92]{bjorner-shellability}
{Anders Bj\"{o}rner},
\emph{The homology and shellability of matroids and geometric lattices},
{Matroid applications},
{Encyclopedia Math. Appl.},
{\bf 40},
{226--283},
{Cambridge Univ. Press, Cambridge},
{(1992)},
\url{https://doi.org/10.1017/CBO9780511662041.008}.


\bibitem[Bry77]{brylawski-broken}
{Tom Brylawski},
\emph{The broken-circuit complex},
{Trans. Amer. Math. Soc.},
{\bf 234},
{(1977)},
{no.~2},
{417--433},
\url{https://doi.org/10.2307/1997928}.


\bibitem[CDFV11]{cohen-denham-falk-varchenko}
{Daniel Cohen, Graham Denham, Michael Falk and Alexander Varchenko},
\emph{Critical points and resonance of hyperplane arrangements},
{Canad. J. Math.},
{Canadian Journal of Mathematics. Journal Canadien de Math\'{e}matiques},
{\bf 63},
{(2011)},
{no.~5},
{1038--1057},
\url{https://doi.org/10.4153/CJM-2011-028-8}.

\bibitem[Cra69]{crapo-tuttepolynomial}
{Henry H. Crapo},
\emph{The {T}utte polynomial},
{Aequationes Math.}, 
{\bf 3},
{(1969)},
{211--229},
\url{https://doi.org/10.1007/BF01817442}.



\bibitem[DGS12]{denham-garrousian-schulze}
{Graham Denham, Mehdi Garrousian and Mathias Schulze},
\emph{A geometric deletion-restriction formula},
{Adv. Math.},
{\bf 230},
{(2012)},
{no.~4-6},
{1979--1994},
\url{https://doi.org/10.1016/j.aim.2012.04.003}.

\bibitem[ELV98]{etienne-lasvergnas}
{Gwihen Etienne and Michel Las Vergnas},
\emph{External and internal elements of a matroid basis},
{Discrete Math.},
{\bf 179},
{(1998)},
{no.~1-3},
{111--119},
\url{https://doi.org/10.1016/S0012-365X(95)00332-Q}.

\bibitem[FS97]{fulton-sturmfels}
{William Fulton and Bernd Sturmfels}, 
\emph{Intersection theory on toric varieties},
{Topology},
{\bf 36},
{(1997)},
{no.~2},
{335--353},
\url{https://doi.org/10.1016/0040-9383(96)00016-X}.



\bibitem[Gio02]{gioan-thesis}
{Emeric Gioan},
\emph{Correspondance naturelle entre bases et r\'eorientations des matr\"oides orient\'e.},
{PhD thesis},
{University of Bordeaux 1, France},
{(2002)}.


\bibitem[GL18]{gioan-lasvergnas}
{Emeric Gioan and Michel Las Vergnas},
\emph{The Active Bijection 2.a - Decomposition of activities for matroid bases, and Tutte polynomial of a matroid in terms of beta invariants of minors},
{arXiv e-prints},
{(2018)},
{arXiv:1807.06516},
\url{https://arxiv.org/abs/1807.06516}.


\bibitem[GT90]{gordon-traldi}
{Gary Gordon and Lorenzo Traldi},
\emph{Generalized activities and the Tutte polynomial},
{Discrete Math.},
{\bf 85},
{(1990)},
{no.~2},
{167--176},
\url{https://doi.org/10.1016/0012-365X(90)90019-E}.

\bibitem[HK12]{huh-katz}
{June Huh and Eric Katz},
\emph{Log-concavity of characteristic polynomials and the Bergman fan of matroids},
{Math. Ann.},
{\bf 354},
{(2012)},
{no.~3},
{1103--1116},
{https://doi.org/10.1007/s00208-011-0777-6}.

\bibitem[Huh12]{huh-chromatic}
{June Huh},
\emph{Milnor numbers of projective hypersurfaces and the chromatic polynomial of graphs},
{J. Amer. Math. Soc.},
{\bf 25},
{(2012)},
{no.~3},
{907--927},
\url{https://doi.org/10.1090/S0894-0347-2012-00731-0}.



\bibitem[Huh13]{huh-likelihood}
{June Huh},
\emph{The maximum likelihood degree of a very affine variety},
{Compos. Math.},
{\bf 149},
{(2013)},
{no.~8},
{1245--1266},
\url{https://doi.org/10.1112/S0010437X13007057}.

\bibitem[Huh14]{huh-thesis}
{June Huh},
\emph{Rota’s conjecture and positivity of algebraic cycles in permutohedral varieties},
{PhD thesis},
{University of Michigan},
{(2014)}.

\bibitem[Huh15]{huh-logconcavity}
{June Huh},
\emph{{$h$}-vectors of matroids and logarithmic concavity},
{Adv. Math.},
{Advances in Mathematics},
{\bf 270},
{2015},
{49--59},
\url{https://doi.org/10.1016/j.aim.2014.11.002}.

\bibitem[JY16]{jensen-yu}
{Anders Jensen and Josephine Yu},
\emph{Stable intersections of tropical varieties},
{J. Algebraic Combin.},
{\bf 43},
{(2016)},
{no.~1},
{101--128},
\url{https://doi.org/10.1007/s10801-015-0627-9}.

\bibitem[KRS99]{kook-reiner-stanton}
{Woong Kook. Victor Reiner, and Dennis Stanton},
\emph{A convolution formula for the Tutte polynomial},
{J. Combin. Theory Ser. B},
{\bf 76},
{(1999)},
{no.~2},
{297--300},
\url{https://doi.org/10.1006/jctb.1998.1888}.

\bibitem[LdMRS20]{medrano-rincon-shaw}
{Luc\'{\i}a L\'opez de Medrano, Felipe  Rinc\'{o}n, and Kristin Shaw},
\emph{Chern-Schwartz-MacPherson cycles of matroids},
{Proc. Lond. Math. Soc. (3)},
{\bf 120},
{(2020)},
{no.~1},
{1--27},
\url{https://doi.org/10.1112/plms.12278}.

\bibitem[MS15]{maclagan-sturmfels}
{Diane Maclagan and Bernd Sturmfels},
\emph{Introduction to tropical geometry},
{Graduate Studies in Mathematics},
{\bf 161},
{American Mathematical Society, Providence, RI},
{(2015)}.



\bibitem[Pro77]{provan-thesis}
{John Scott Provan},
\emph{Decompositions, shellings, and diameters of simplicial complexes and convex polyhedra},
{Thesis (Ph.D.)--Cornell University},
{ProQuest LLC, Ann Arbor, MI},
{(1977)},
\url{http://search.proquest.com/docview/302843450}.





\bibitem[Wil76]{wilf-chromatic}
{Herbert S. Wilf},
\emph{Which polynomials are chromatic?},
{Colloquio Internazionale sulle Teorie {C}ombinatorie ({R}oma, 1973), {T}omo {I}},
{247--256. Atti dei Convegni Lincei, No. 17},
{(1976)}.



\end{thebibliography}
\end{document}